\documentclass[10pt
]
{article}
\usepackage{bbm}
\usepackage{dsfont}
\usepackage{color}
\usepackage{xcolor}
\usepackage{mathrsfs}
\usepackage{amssymb}
\usepackage{calligra}
\usepackage{fontenc}
\usepackage{amsbsy}
\usepackage{amsmath}
\usepackage{
amsthm}
\usepackage{latexsym}
\usepackage[mathscr]{eucal}

\newcommand{\Ln}{[\!|}
\newcommand{\Rn}{|\!]}
\newcommand{\n}[1]{{\Ln{#1}\Rn}}
\newtheorem{theorem}{Theorem}
\newtheorem{lemma}{Lemma}
\newtheorem{proposition}{Proposition}

\newtheorem{remark}{Remark}
\newtheorem{definition}{Definition}
\numberwithin{equation}{section}
\numberwithin{theorem}{section}
\numberwithin{lemma}{section}
\numberwithin{proposition}{section}
\numberwithin{corollary}{section}
\numberwithin{remark}{section}
\numberwithin{definition}{section}
\newcommand{\Div}{\mbox{\rm div}\,}
\newcommand{\be}{\begin{equation}}
\newcommand{\ba}{\begin{array}}
\newcommand{\ea}{\end{array}}
\newcommand{\ee}{\end{equation}}
\newcommand{\Lim}[1]{{\displaystyle \lim_{ #1}}}
\newcommand{\real}{{\mathbb R}}
\newcommand{\eeq}[1]{\label{eq:#1}\end{equation}}
\newcommand{\Eqref}[1]{{\rm (\ref{eq:#1})}}
\newcommand{\Br}{\begin{remark}\begin{rm}}
\newcommand{\ER}[1]{\end{rm}\label{remark:#1}\end{remark}}

\begin{document}
\title{The Transition Problem between Time-Independent Motions\\ of a Body in a Viscous Liquid}
\author{Giovanni P. Galdi and Toshiaki Hishida}
\date{}
\maketitle
\begin{abstract}
A body $\mathscr B$ moves in an unbounded Navier-Stokes liquid by time-independent translatory motion. Suppose that at time $t=0$, $\mathscr B$ smoothly changes its motion to an {\em arbitrary} rigid motion, reached at time $t=1$.
We then show that the associated Navier-Stokes  problem has a unique solution connecting the two steady-states generated by the motion of $\mathscr B$, provided all the involved velocities of $\mathscr B$ are sufficiently small. 
\end{abstract}
\par\noindent
{\small {\bf Mathematics Subject Classification.} 35Q30, 76D05.}\smallskip\par\noindent
{\small {\bf Keywords.} Navier--Stokes flow past a body,  transition problem,  steady flow, attainability, asymptotic behavior in time, Oseen evolution operator.}

\section{Introduction}        
\label{intro}
In 1965, Robert Finn posed the question of providing a rigorous mathematical proof that a (finite) body $\mathscr B$, completely immersed in a viscous Navier-Stokes liquid that fills the
three-dimensional
whole space, $\Omega$, outside $\mathscr B$, can be accelerated from rest to a state of motion characterized by a constant translational velocity $\gamma$ \cite{Finn}. This is the so-called ``starting problem." Basically, it consists in showing 
the convergence of nonstationary motions of the liquid, generated by the acceleration of $\mathscr B$, to the time independent (steady-state) motion  
corresponding to $\gamma$.
\par
The main difficulty in answering  this question is due to the fact that the classical ($L^2$) energy estimate --fundamental in establishing global-in-time results-- is not available in such a  case, because $\Omega$ is an exterior domain. It   suggests, instead, that one should  resort to a suitable  $L^q$ approach, with $q\ge 3$. An extended theory of this kind was successfully obtained by Shibata about thirty years later \cite{Shibata} (see also \cite{KoShi}), which, short after, led to the complete solution of Finn's problem, at least for a ``small" $\gamma$ \cite{GHS97}.\footnote{It is highly unlikely that the problem can be solved for ``large'' $\gamma$, unless one adds, for example, appropriate control forces on the body.}
\par
Over time, the results in \cite{GHS97} have been improved and generalized in several respects; see, e.g., \cite{GH21,HM,Saz,Ta21,Ta22,Ta24}. Of particular relevance to our investigation is the question addressed by Sazonov in \cite{Saz}. Specifically, he is interested in the more general situation where the initial state of $\mathscr B$ is not necessarily rest, but a generic translatory motion with velocity $\gamma_0$  (say). Sazonov refers to this as  ``transition problem." In \cite[Theorems 3.1 and 3.3]{Saz}, he is able to solve this problem, among others, provided that $\gamma_0$ and $\gamma$ are parallel,  $|\gamma-\gamma_0|$  is ``small" and the steady-state corresponding to $\gamma$ is stable, in the sense of spectral theory. However, the crucial point that makes such a result rather unsatisfactory is that it is proved under the further assumption  $\Omega\equiv\mathbb R^3$, that is, in the absence of the body. 
\par
Objective of this paper is to give a positive answer to the transition problem, in a somewhat general formulation. More specifically, we assume that the initial state is generated by $\mathscr B$ that translates with constant velocity $\gamma_1$, while, in the final state, $\mathscr B$ is allowed to perform an {\em arbitrary} rigid motion with velocity ${\sf V}$, characterized by (constant) translation,  $\gamma_2$, and rotation, $\omega_0$. No assumptions are made about the direction of $\gamma_i$, $i=1,2$, and $\omega_0$, but only about their magnitude which is assumed to be less than a given constant. Let $v_i$, $i=1,2$ be the (velocity fields of the) steady-state solutions corresponding to $\gamma_1$ and ${\sf V}$, respectively. We then show that, in a suitable function class, there exists a unique solution whose velocity field coincides with $v_1$ at time $t=0$ and, as $t\to\infty$, converges to $v_2$, with a distinct order of decay.
\par
The starting point of our analysis is to rewrite the original set of equations as a single integral equation in an appropriate function space (weak-$L^3$ space); see \eqref{IE}--\eqref{reduced}. This is done  employing the evolution operator $T=T(t,s)$ constructed in \cite{HR14}. It is then easy to recognize that, formally, the transition problem is equivalent to finding sufficient conditions that guarantee that the solutions to this integral equation converge to 0 as $t\to\infty$. To make this argument rigorous, we begin to give a precise definition of global (in time) solution (see Definition \ref{def-sol}). Successively, in Theorem \ref{main}, we prove that,  if the data $\gamma_i$, $i=1,2,$ and $\omega_0$ are below a certain constant involving the initial acceleration of $\mathscr B$, then  solutions exist, are  unique, depend continuously upon the data and decay to 0, as $t\to\infty$ in different norms, with corresponding (algebraic) decay rate; see \eqref{decay}. The proof of this result is obtained by the contraction mapping theorem. For its success, we need, basically, three  fundamental estimates: the first one, regarding the ``forcing term" $g$, the second one involving linear terms with coefficients depending on $v_2$, and the third one  related to the nonlinear term. The evaluation of the first two terms   can be deduced from the known results on the
large time behavior of the evolution
operator $T(t,s)$ proved in
{\color{black}
\cite{GH21,Hi18, Hi20}
}
and recalled in
{\color{black}
Propositions \ref{prop-evo1} and \ref{prop-evo2},
}
provided $g$ and $v_2$ meet appropriate global summability properties  along with corresponding estimates with regard to the data; see Lemmas \ref{lem-Psi} and \ref{lem-top}. Such properties are indeed shown in Lemma  \ref{lem-force} under the given assumption, namely, that $\mathscr B$ goes  from an initial  translatory motion to an arbitrary rigid motion. It is worth emphasizing that similar estimates are not known if, in the initial state, the angular velocity of $\mathscr B$ is not zero, in which case the transition problem remains open;
see Remark \ref{rem-force}.
We wish also to remark that the summability properties of $g$ and their dependence on the data may as well affect the asymptotic behavior in time of the solution; see  Remark \ref{rem-1}. The estimate of the term in the integral equation involving the nonlinearity (see \eqref{Psi}) is performed in Lemma \ref{lem-Psi}, also with the help of the results reported in Propositions \ref{prop-evo1} and  \ref{prop-evo2}. Our objective here is twofold. On the one hand, to show quadratic bounds, in suitable norms, that allow the use of a contraction-mapping argument. On the other hand, to prove, again in suitable norms, algebraic  decay estimates in time. Finally, with the help of Lemma \ref{lem-uni} we provide the uniqueness property.
\par          
The paper is organized as follows. In Section \ref{formulation} we furnish the precise mathematical formulation of our transition problem. In the following Section \ref{result}, after introducing some standard notation, we show that the problem can be equivalently rewritten as an integral equation. We thus give the definition of a solution to this equation  and state our main results in Theorem \ref{main}. We also furnish a few remarks about this theorem and some of its consequences. Then, In Section \ref{pre}, we  collect several preparatory results, mostly, from \cite{GH21,Hi18,Hi20}, and,  in the last Section \ref{proof}, we give the proof of Theorem \ref{main}.

\section{Formulation of the Problem}\label{formulation}
Consider a rigid body $\mathscr B$  (compact subset of $\mathbb R^3$) moving in a Navier-Stokes liquid, $\mathscr L$, that fills the entire space, $\Omega$, outside $\mathscr B$. Up to time $t=0$ (say) $\mathscr B$ moves with a time-independent translational motion characterized by the constant velocity $\gamma_1$, while the flow of $\mathscr L$, referred to a body-fixed frame, $\mathcal F$, is steady and characterized by velocity and pressure fields $v_1$ and $p_1$, respectively.  Successively, in the time interval
$(0,1)$,
$\mathscr B$ performs an appropriate and given rigid motion such that, at time $t_1$, its motion is still time-independent but arbitrary, and described by (constant) translational velocity $\gamma_2$ (in general, $\neq\gamma_1$) and angular velocity $\omega_0$. We denote by $v_2$ and $p_2$ velocity and pressure fields of  a steady flow of $\mathscr L$ corresponding to
$\gamma_2+\omega_0\times x$. 
The problem we want to investigate is whether the unsteady flow of $\mathscr L$, generated in the time interval 
$(0,1)$,
will converge to such a steady flow in the limit $t\to\infty$.   
\par
In order to formulate this problem in a precise mathematical way, we begin to observe that ${\sf s}_i:=(v_i,p_i)$, $i=1,2$, solve the following set of equations
$$\begin{array}{cc}\medskip\left.\begin{array}{ll}\medskip
\Delta v_1 +\gamma_1\cdot\nabla v_1=v_1\cdot\nabla v_1+\nabla p_1\\
\Div v_1=0\end{array}\right\}\ \ \mbox{in $\Omega$}\\
v_1(x)=\gamma_1\ \ \mbox{at $\partial\Omega$}\,;\ \ \displaystyle{\lim_{|x|\to\infty} v_1(x)=0}\,,
\end{array}
$$
and
\be\ba{cc}\medskip\left.\ba{ll}\medskip
\Delta v_2 +{\sf V}\cdot\nabla v_2-\omega_0\times v_2=v_2\cdot\nabla v_2+\nabla p_2\\
\Div v_2=0\ea\right\}\ \ \mbox{in $\Omega$}\\
v_2(x)={\sf V}\ \ \mbox{at $\partial\Omega$}\,;\ \ \Lim{|x|\to\infty} v_2(x)=0\,,
\ea
\eeq{1}
where
$$
{\sf V}=\gamma_2+\omega_0\times x\,.
$$
Next, let $\gamma=\gamma(t)$ and $\omega=\omega(t)$ be the prescribed  translational and angular velocity of $\mathscr B$ for $t>0$. By what we said above, they must satisfy
\be
\gamma(0)=\gamma_1\,,\ \ \omega(0)=0\,;\ \ \gamma(t)=\gamma_2\,,\ \ \omega(t)=\omega_0\,,\ \ \mbox{for all 
{\color{black}
$t\ge 1$
}}. 
\eeq{2}
Thus, with ${V}(x,t):=\gamma(t) +\omega(t)\times x$, the equations of motion of $\mathscr L$ in $\mathcal F$ read as follows
\be\ba{cc}\medskip\left.\ba{ll}\medskip
\partial_t v+( v-{V})\cdot\nabla v+\omega\times v=\Delta v -\nabla p\\
\Div v=0\ea\right\}\ \ \mbox{in $\Omega\times(0,\infty)$}\\ \medskip
 v(x,t)={V}(x,t)\ \ \mbox{at $\partial\Omega\times(0,\infty)$}\,;\ \ \Lim{|x|\to\infty} v(x,t)=0\,,\\
 v(x,0)= v_1(x)\,.
\ea
\eeq{3}
The transition problem can be then formulated as follows: {\it Find a solution $(v,p)$ to \Eqref{2}--\Eqref{3} such that, in appropriate norm,}
\be
\lim_{t\to\infty}v(x,t)= v_2(x)\,,\ \ x\in\Omega\,.
\eeq{4}
\par
In order to solve this problem, we have to specify the way in which the transition ${\sf s}_1\to{\sf s}_2$ occurs, namely, furnish an explicit realization of conditions \Eqref{2}. To this end, let $\psi=\psi(t)$, $t\ge 0$, be a smooth non-decreasing function such that $\psi(0)=0$ and $\psi(t)=1$, for all $t\ge 1$, and 
{\color{black}
we have the case
}
$\omega_0\in\real^3\backslash\{0\}$
in mind (although the other case $\omega_0=0$ is not excluded in our main result, see Remark \ref{rem-6}).
We then set
\be
\gamma(t)=(1-\psi(t))\gamma_1+\psi(t)\gamma_2\,,\ \ \omega(t)=\psi(t)\omega_0
\,,
\eeq{5}
and look for a solution to \Eqref{3}--\Eqref{5} of the form
$$
v(x,t)=u(x,t)+\psi(t) v_2(x)\,,\ \ p(x,t)={\sf p}(x,t)+\psi(t)p_2(t)\,.
$$
From \Eqref{3}, \Eqref{4}, and \Eqref{1}, we deduce that the pair $( u,{\sf p})$ solves the following problem
\be\ba{cc}\medskip\left.\ba{ll}\medskip
\partial_t u+( u-{V})\cdot\nabla u+\omega\times u+\psi\big( v_2\cdot\nabla u+ u\cdot\nabla v_2\big)=\Delta u -\nabla {\sf p}+f\\
\Div u
=0\ea\right\}\ \ \mbox{in $\Omega\times(0,\infty)$}\\ \medskip
u(x,t)=(1-\psi(t))\gamma_1\ \ \mbox{at $\partial\Omega\times(0,\infty)$}\,;\ \ \Lim{|x|\to\infty} u(x,t)=\Lim{t\to\infty} u(x,t)=0\,,\\
u(x,0)= v_1(x)\,,
\ea
\eeq{6}
where
\be
f:=
-\psi^\prime v_2+(\psi-\psi^2)\big[(v_2+\gamma_1-\gamma_2-\omega_0\times x)\cdot\nabla v_2+\omega_0\times v_2\big]\,.
\eeq{7}
The transition problem reduces then to find a solution to \Eqref{6}--\Eqref{7}.
\Br Finn's ``starting problem" is a particular case of the ``transition problem," obtained by setting in \Eqref{6}--\Eqref{7} $\gamma_1\equiv
\omega_0\equiv0$.
\ER{1}
\Br Of course, \Eqref{5}  is only a possible choice of the way in which   the transition may occur, even though rather reasonable.
\ER{2}

\section{Statement of the main result}
\label{result}

Let us begin with introducing notation.
Given two vector fields $v$ and $w$, we denote by $v\otimes w$ the matrix $(v_iw_j)$.
Let $A=(A_{ij}(x))$ be a $3\times 3$ matrix valued function, then the vector field $\mbox{div $A$}$ is defined by
$(\mbox{div $A$})_i=\sum_j\partial_jA_{ij}$.
By following this rule, $\mbox{div $w$}=0$ implies that $w\cdot\nabla v=\mbox{div $(v\otimes w)$}$.

Let $\Omega$ be an exterior domain in $\mathbb R^3$ with $C^2$-boundary $\partial\Omega$ satisfying
\begin{equation}
\mathbb R^3\setminus\Omega \subset B_1,
\label{body}
\end{equation}
where $B_R$ denotes the open ball centered at the origin with radius $R>0$.
Given $q\in [1,\infty]$
and a domain $G\subset \mathbb R^3$,
the norm of the Lebesgue space $L^q(G)$ is denoted by $\|\cdot\|_{q,G}$.
We abbreviate $\|\cdot\|_q=\|\cdot\|_{q,\Omega}$ for the exterior domain $\Omega$ under consideration.
Given an integer $k>0$,  $W^{k,q}(\Omega)$ stands for the standard $L^q$-Sobolev space
with the norm $\|\cdot\|_{k,q}$.
The class $C_0^\infty(\Omega)$ consists of all $C^\infty$ functions with compact support in $\Omega$,
then $W^{k,q}_0(\Omega)$ denotes the completion of $C_0^\infty(\Omega)$ in $W^{k,q}(\Omega)$. 

Let us also introduce the Lorentz spaces which should be defined as Banach spaces by using the average function of the 
nonincreasing rearrangement, see \cite{BL} for details.
For simplicity, in this paper, we define those spaces just by
\[
L^{q,r}(\Omega)=\big(L^1(\Omega), L^\infty(\Omega)\big)_{\theta,r} \qquad\mbox{with}\quad \theta=1-\frac{1}{q}
\]
for $q\in (1,\infty)$ and $r\in [1,\infty]$ via the real interpolation functor $(\cdot,\cdot)_{\theta,r}$.
Then the reiteration theorem in the interpolation theory implies that
\begin{equation}
L^{q,r}(\Omega)=\big(L^{q_0}(\Omega), L^{q_1}(\Omega)\big)_{\theta,r}
\label{lorentz}
\end{equation}
whenever
\begin{equation}
1<q_0<q<q_1<\infty, \qquad
\frac{1}{q}=\frac{1-\theta}{q_0}+\frac{\theta}{q_1}, \qquad
r\in [1,\infty].
\label{inter}
\end{equation}
We denote by $\|\cdot\|_{(q,r)}$ the norm of the Lorentz space $L^{q,r}(\Omega)$.
Notice that $L^{q,q}(\Omega)=L^q(\Omega)$.
We have the duality relation
$L^{q^\prime,r^\prime}(\Omega)=L^{q,r}(\Omega)^*$ provided that
\begin{equation}
q\in (1,\infty), \qquad
r\in [1,\infty), \qquad
\frac{1}{q^\prime}+\frac{1}{q}=1, \qquad
\frac{1}{r^\prime}+\frac{1}{r}=1.
\label{duality}
\end{equation}
By $\langle\cdot,\cdot\rangle$ we denote various duality pairings over the exterior domain $\Omega$.

To describe the class of steady state, we need the homogeneous $L^2$-Sobolev space 
$D^{k,2}(\Omega)$ consisting of all functions $u\in L^1_{\rm loc}(\Omega)$ satisfying $\nabla^ku\in L^2(\Omega)$,
where $k=1,2$.
Here and in what follows we adopt the same symbol for denoting scalar, vector and tensor function spaces
as long as there is no confusion.

Let us introduce the solenoidal function space.
The class $C^\infty_{0,\sigma}(\Omega)$ consists of all divergence-free vector fields being in $C_0^\infty(\Omega)$.
Let $q\in (1,\infty)$.
We denote by $L^q_\sigma(\Omega)$ the completion of $C^\infty_{0,\sigma}(\Omega)$ in $L^q(\Omega)$.
The space of $L^q$-vector fields admits the Helmholtz decomposition
\[
L^q(\Omega)=L^q_\sigma(\Omega)\oplus \{\nabla p\in L^q(\Omega); p\in L^q_{\rm loc}(\overline{\Omega})\},
\]
see Miyakawa \cite{Mi82} and Simader and Sohr \cite{SiSo}.
By $P=P_q: L^q(\Omega)\to L^q_\sigma(\Omega)$ we denote the Fujita-Kato projection associated with the Helmholtz decomposition
above.
By \eqref{lorentz} the projection $P_q$ extends to a bounded operator $P_{q,r}$ on the Lorentz space $L^{q,r}(\Omega)$
for every $ q\in (1,\infty)$ and $r\in [1,\infty]$.
Following Borchers and Miyakawa \cite{BoMi}, let us define the solenoidal Lorentz space by
$L^{q,r}_\sigma(\Omega):=R(P_{q,r})$, that is the range of $P_{q,r}$.
Then it is characterized as
\[
L^{q,r}_\sigma(\Omega)=\{u\in L^{q,r}(\Omega); \mbox{div $u$}=0,\, \nu\cdot u|_{\partial\Omega}=0\},
\]
where $\nu$ stands for the outer unit normal to $\partial\Omega$.
We have also
the duality relation $L^{q^\prime,r^\prime}_\sigma(\Omega)=L^{q,r}_\sigma(\Omega)^*$ for $(q,r)$ with the condition \eqref{duality}.
Moreover, 
\begin{equation}
L^{q,r}_\sigma(\Omega)=\big(L^{q_0}_\sigma(\Omega), L^{q_1}_\sigma(\Omega)\big)_{\theta,r}
\label{lorentz2}
\end{equation}
provided that \eqref{inter} is satisfied.
See \cite[Theorems 5.2 and 5.4]{BoMi} for these results.
Finally, we denote several positive constants by $C$, which may change from line to line.

Let $q\in (1,\infty)$, then we introduce the linear operator
\begin{equation}
\left\{
\begin{array}{l}
D(L(t))=\{u\in W^{2,q}(\Omega)\cap W^{1,q}_0(\Omega)\cap L^q_\sigma(\Omega); (\omega_0\times x)\cdot\nabla u\in L^q(\Omega)\}, \\
L(t)u=-P[\Delta u+V(\cdot,t)\cdot\nabla u-\omega(t)\times u],
\end{array}
\right.
\label{linerized}
\end{equation}
where, we recall,
\begin{equation}
\begin{split}
&V(x,t)=\gamma(t)+\omega(t)\times x, \\
&\gamma(t)=(1-\psi(t))\gamma_1+\psi(t)\gamma_2, \qquad
\omega(t)=\psi(t)\omega_0,
\end{split}
\label{rigid}
\end{equation}
with constant vectors $\gamma_1, \gamma_2, \omega_0\in\mathbb R^3$
and
\begin{equation}
\psi\in C^1(\mathbb R; [0,1]), \qquad \psi(t)=0\quad (t\leq 0), \qquad \psi(t)=1\quad (t\geq 1).
\label{trans}
\end{equation}
We know from Hansel and Rhandi \cite{HR14} that the family $\{L(t); t\geq 0\}$ generates an evolution operator
$\{T(t,s); t\geq s\geq 0\}$ on the space $L^q_\sigma(\Omega)$ for every $q\in (1,\infty)$.
They also derived the
$L^q$-$L^r$ smoothing estimates of $\nabla^j T(t,s)$ for $j=0,1$. 
Since the issue of the present paper is the large time behavior, what we need
is the $L^q$-$L^r$ decay estimates, see \cite{Hi18, Hi20}, 
for its adjoint $T(t,s)^*$ as well as $T(t,s)$.
We recall some of those estimates 
in Propositions \ref{prop-evo1} and \ref{prop-evo2} below.

As a lifting function of the inhomogeneous boundary value of $u(x,t)$, one can take
\begin{equation}
b(x,t)=(1-\psi(t))\,\mbox{rot}\left(\phi(x)\frac{\gamma_1\times x}{2}\right)
\label{lift}
\end{equation}
with $\phi\in C_0^\infty(B_2)$ being fixed such that $\phi(x)=1$ in $B_1$.
Indeed, in view of \eqref{body},
we have
\[
b|_{\partial\Omega}=(1-\psi(t))\gamma_1, \qquad \mbox{div $b$}=0.
\]
Let us look for a solution $u(x,t)$ to \Eqref{6} of the form
\[
u(x,t)=w(x,t)+b(x,t)
\]
and, therefore,
\be
v(x,t)=w(x,t)+b(x,t)+\psi(t)v_2(x).
\eeq{vsol}
By use of the evolution operator $T(t,s)$, our problem 
$\lim_{t\to\infty}v(t)=v_2$
is reduced to the deduction of
the large time decay of $w(t)$ obeying the integral equation
\begin{equation}
w(t)=T(t,0)v_0+\int_0^tT(t,s)Pg(s)\,ds+\int_0^tT(t,s)P\mbox{div $(Fw)$}(s)\,ds,
\label{IE}
\end{equation}
where
\begin{equation}
\begin{split}
&v_0=v_1-b(\cdot,0), \\
&g=f-\partial_tb-b\cdot\nabla b+\Delta b+V\cdot\nabla b-\omega\times b-\psi(v_2\cdot\nabla b+b\cdot\nabla v_2), \\
&Fw=w\otimes w+w\otimes (b+\psi v_2)+(b+\psi v_2)\otimes w\,,
\end{split}
\label{reduced}
\end{equation}
and $f$ is defined in \Eqref{7}.
Suitable smallness conditions on $\gamma_1$ and on $(\gamma_2,\omega_0)$ allow us to show the
existence of a unique steady state $v_1$ and $v_2$, respectively, with their fine properties.
See the assumption of Lemma \ref{lem-top} below.
In what follows we often use
\begin{equation}
\|v_0\|_3\leq \|v_1\|_3+C|\gamma_1|\leq C|\gamma_1|^{1/2},
\label{data-1}
\end{equation}
\begin{equation}
\sup_{0\leq t\leq 1}\|g(t)\|_{(3,\infty)}\leq C(1+|\psi^\prime|_0)
\big(|\gamma_1|^{1/2}+|\gamma_2|+|\omega_0|\big),
\label{data-2}
\end{equation}
\begin{equation}
\|b+\psi v_2\|_{(r,\infty)}\leq C_r\big(|\gamma_1|^{1/2}+|\gamma_2|+|\omega_0|\big), \qquad r\in [3,\infty],
\label{data-3}
\end{equation}
\begin{equation}
\|\nabla (b+\psi v_2)\|_r\leq C_r\big(|\gamma_1|^{1/2}+|\gamma_2|+|\omega_0|\big), \qquad r\in [2,6],
\label{data-4}
\end{equation}
where
\begin{equation}
|\psi^\prime|_0:=\sup_{0\leq t\leq 1}\left|\frac{d\psi}{dt}(t)\right|,
\label{size-psi}
\end{equation}
and $\|\cdot\|_{(\infty,\infty)}=\|\cdot\|_\infty$.
In \eqref{data-1}, estimate of the steady state $v_1$ with wake structure was studied by several authors; among them, 
Takahashi \cite[Theorem 1.1]{Ta21} made use of $L^q$-estimates of the Oseen system developed by 
the first author \cite{G92, G-b} to deduce
$\|v_1\|_3\leq C|\gamma_1|^{1/2}$ for small $|\gamma_1|$, where the power $\frac{1}{2}$
is determined by the linear theory.
One may assume $|\gamma_1|\leq 1$, so that $|\gamma_1|$ is replaced by $|\gamma_1|^{1/2}$ 
in \eqref{data-1}--\eqref{data-4}
for simplicity.
Estimate \eqref{data-2} follows from Lemma \ref{lem-force}, that we will show in the next section
under the smallness of $|\gamma_2|+|\omega_0|$, together with
\[
\|v_2\cdot\nabla b+b\cdot\nabla v_2\|_{(3,\infty)}
\leq \|v_2\|_{(3,\infty)}\|\nabla b\|_\infty+\|b\|_6\|\nabla v_2\|_6
\leq C|\gamma_1|\big(|\gamma_2|+|\omega_0|\big).
\]
Recall that $g(t)=0$ for $t\geq 1$, however, 
$g(t)\notin L^3(\Omega)$ for $t<1$ in general when $\gamma_2\cdot \omega_0=0$
{\color{black}
and $\omega_0\neq 0$.
}

Since $(s,\infty)\ni t\mapsto T(t,s)h$ with values in $L^{3,\infty}_\sigma(\Omega)$ is continuous only in the weak-$*$ sense 
when $h\in L^{3,\infty}_\sigma(\Omega)$,
the regularity of the second term of the right-hand side of \eqref{IE} that one could expect is \eqref{sol-cl} below,
see \eqref{top-conti-1} and Remark \ref{rem-conti} in section \ref{proof}.
Having this in mind,
we introduce the definition of solutions to \eqref{IE}.
\begin{definition}
We say that $w(t)$ is a global solution to \eqref{IE} if

\begin{enumerate}
\item
it is of class
\begin{equation}
w\in C_{w^*}((0,\infty); L^{3,\infty}_\sigma(\Omega))
\label{sol-cl}
\end{equation}
with the initial condition
\begin{equation}
\lim_{t\to 0}\|w(t)-v_0\|_{(3,\infty)}=0,
\label{IC}
\end{equation}

\item
the third term on the right-hand side of \eqref{IE} is Bochner integrable in $L^{3,\infty}_\sigma(\Omega)$,

\item
\eqref{IE} is satisfied in $L^{3,\infty}_\sigma(\Omega)$ for every $t>0$.
\end{enumerate}
\label{def-sol}
\end{definition}

In order to ensure 
the Bochner integrability in the item 2 of Definition \ref{def-sol}, one needs a bit more regularity than \eqref{sol-cl}.
If, for instance, $w$ is assumed to belong to the auxiliary space $w\in L^\infty_{\rm loc}(0,\infty; L^{q,\infty}_\sigma(\Omega))$
for some $q\in (3,\infty)$ with 
$\|w(t)\|_{(q,\infty)}
=O(t^{-1/2+3/2q})$ as $t\to 0$,
then the item 1 of Lemma \ref{lem-Psi} tells us that
the third term of the right-hand side of \eqref{IE} is indeed Bochner integrable even in $L^3_\sigma(\Omega)$.

The reason why one may expect the strong convergence
\eqref{IC} is that $v_0\in L^3_\sigma(\Omega)$ due to the wake structure of $v_1$.
If we started from a steady flow $v_1$ corresponding to a purely rotation, then 
the vector field $v_0\in L^{3,\infty}_\sigma(\Omega)$ given by \eqref{reduced} would 
yield merely the weakly-$*$ convergence instead of \eqref{IC},
see also Remark \ref{rem-force}.

The main result now reads as follows.
\begin{theorem}
Let $\gamma_1, \gamma_2, \omega_0\in \mathbb R^3$, and set
\begin{equation}
D:=
|\gamma_1|^{1/2}+|\gamma_2|+|\omega_0|.
\label{D}
\end{equation}
For every $\varepsilon\in (0,\frac{1}{4})$, there is a constant $\delta=\delta(\varepsilon,\Omega)\in (0,1]$
independent of $|\psi^\prime|_0$ such that if
\begin{equation}
D\leq \frac{\delta}{1+|\psi^\prime|_0},
\label{small}
\end{equation}
where $|\psi^\prime|_0$ is given by \eqref{size-psi}, see also \eqref{trans},
then problem \eqref{IE} admits a global solution $w(t)$ of class
\begin{equation}
w\in C((0,\infty); L^r_\sigma(\Omega)), \qquad 
w\in C_{w^*}((0,\infty); L^\infty(\Omega))
\label{sol-cl2}
\end{equation}
\begin{equation}
\nabla w\in C_{w^*}((0,\infty); L^{3,\infty}(\Omega))
\label{sol-cl3}
\end{equation}
for all $r\in (3,\infty)$
as well as \eqref{sol-cl}, which enjoys
\begin{equation}
\|w(t)\|_{(3,\infty)}\leq C(1+|\psi^\prime|_0)D 
\label{bdd}
\end{equation}
for all $t\geq 0$ with some constant $C=C(\varepsilon,\Omega)>0$ and
\begin{equation}
\begin{split}
&\|w(t)\|_r=\left\{
\begin{array}{ll}
O(t^{-1/2+3/2r}), \qquad &r\in (3,q_0), \\
O(t^{-1/2+\varepsilon}\log t), &r=q_0, \\
O(t^{-1/2+\varepsilon}), &r\in (q_0,\infty],
\end{array}
\right.  \\
&\|w(t)\|_{(q_0,\infty)}
+\|\nabla w(t)\|_{(3,\infty)}
=O(t^{-1/2+\varepsilon}),
\end{split}
\label{decay}
\end{equation}
as $t\to \infty$, where $q_0:=\frac{3}{2\varepsilon}$.

Suppose that $\widetilde w(t)$ is another global solution to \eqref{IE} of class
\[
\widetilde w\in L^\infty_{\rm loc}(0,\infty; L^{r,\infty}_\sigma(\Omega))
\]
with some $r\in (3,\infty)$. 
Then there is a constant $\widetilde\delta=\widetilde\delta(\Omega)\in (0,1]$
independent of $r,\,\varepsilon$
and $|\psi^\prime|_0$ such that if
\begin{equation}
D\leq\frac{\widetilde\delta}{1+|\psi^\prime|_0}
\label{small2}
\end{equation}
as well as \eqref{small}, then
$\widetilde w(t)$ coincides with
the solution $w(t)$ obtained above. 
\label{main}
\end{theorem}

Several remarks are in order.
\begin{remark}
If, in particular, $\gamma_2\cdot\omega_0\neq 0$, then a small steady state $v_2\in L^q(\Omega)$ with $q>2$ 
having the wake structure is available
through the Mozzi-Chasles transform \cite{G-b, GS07}, see Galdi and Silvestre \cite{GS07-k}.
In this case, as in the paper \cite[Theorem 1.2]{Ta21} by Tomoki Takahashi on Finn's starting problem,
see also \cite[Theorem 2.13]{Ta24},
it is possible to deduce better decay
$\|w(t)\|_\infty=O(t^{-3/2q})$ 
under an appropriate smallness conditions on $\|v_2\|_q$ and $\|\nabla v_2\|_r$ with $r>4/3$
as well as $|\gamma_1|$.
However, in view of the $L^q$-theory developed by \cite{F05, GK13},
those norms of $v_2$ can be no longer controlled in terms of $|\gamma_2|+|\omega_0|$.
Thus, the aforementioned large time behavior with faster rate is not included in Theorem \ref{main}.
\label{rem-1}
\end{remark}
\begin{remark}
We would conjecture that $\|w(t)\|_\infty=O(t^{-1/2})$ as $t\to\infty$ in \eqref{decay}, 
which was shown by \cite{HM} on the starting problem
only with translation, however, this desired behavior is still open.
In the context of stability of the steady flow
belonging to $L^{3,\infty}(\Omega)$, the $L^\infty$-decay 
$\|w(t)\|_\infty=O(t^{-1/2+\varepsilon})$ was derived first by Koba \cite{K17}, 
see also \cite[Remark 4.4]{Ta22}.
Later on, his proof has been considerably refined by Takahashi \cite{Ta22} and by the present authors \cite{GH21}, respectively.
Estimates of the adjoint evolution operator $T(t,s)^*P\varphi$ with $\varphi\in C_0^\infty(\Omega)$ 
in terms of $\|\varphi\|_1$ are utilized
in the former paper, while the latter one makes use of the $L^\infty$-estimate of the composite operator
$T(t,s)P\,\mbox{\rm div}$, see Proposition \ref{prop-evo2}.
We follow the latter in this paper.
The rate of $L^{q_0}$-decay 
$\|w(t)\|_{q_0}=O(t^{-1/2+\varepsilon}\log t)$ in \eqref{decay}
was already discovered by \cite[Remark 2.2]{Ta22}.
\label{rem-2}
\end{remark}
\begin{remark}
As in \cite[Subsection 3.2]{Hi22}, 
it is possible to prove further regularity
\begin{equation}
w(t)\in D_r(A^{1/2})\subset W^{1,r}_0(\Omega), \qquad
\nabla w\in C_w((0,\infty); L^r(\Omega)) 
\label{reg}
\end{equation}
for all $t>0$ and $r\in (3,\infty)$
by identifying the global solution obtained in Theorem \ref{main}
with a local solution having the desired regularity
reconstructed in a neighborhood of each time,
where $A$ is the Stokes operator on $L^r_\sigma(\Omega)$ with domain $D_r(A)$.
This argument is due to Kozono and Yamazaki \cite{KY98}. 
The reason why $A^{1/2}$ is used is the verification of 
$w=0$ at $\partial\Omega$ 
in the sense of trace.
Although \eqref{sol-cl2} is also shown together with \eqref{reg} within this reconstruction procedure as in \cite{Hi22, Ta22},
we take the other way in this paper, that is, we will show \eqref{sol-cl2}
as well as \eqref{sol-cl3}
directly. 
In addition, we also deduce the decay property of $\nabla w(t)$ in $L^{3,\infty}(\Omega)$, 
the rate of which is comparable to the one of $w(t)$ in $L^\infty(\Omega)$ (although it is less sharp), see \eqref{decay},
and which is not found in \cite{GH21, HS09, K17, KY98, Ta22} on stability/attainability of basic flows being merely
in the scale-critical space $L^{3,\infty}(\Omega)$.
The same argument as in deduction of the decay of $\|\nabla w(t)\|_{(3,\infty)}$
with the aid of \eqref{evo-est13} below for $r>3$ leads us to the asymptotic behavior
$\|\nabla w(t)\|_r=O(t^{-1/2+\varepsilon})$ as $t\to\infty$ for every $r\in (3,\infty)$.
\label{rem-3}
\end{remark}
\begin{remark}
Since the evolution operator $T(t,s)$ is not of parabolic type in the sense of Tanabe \cite{T}
(unless the rotation is absent), the regularity issue is highly nontrivial.
Asami and the second author \cite{AH} have recently developed the regularity theory of the evolution operator
$T(t,s)$ to establish a new existence theorem for a local solution that is $C^1$ in time with values in 
$L^q_\sigma(\Omega)$ for some $q\in (3,\infty)$ if assuming $\rho w(0)\in L^q(\Omega)$ in addition to 
$w(0)\in L^q_\sigma(\Omega)$ and if the external force is absent, where $\rho(x)=1+|x|$.
In view of this theory, 
such a temporal regularity of the global solution obtained in Theorem \ref{main}
seems unlikely even though $\psi^\prime(t)$ is H\"older continuous additionally to \eqref{trans}.
\label{rem-4}
\end{remark}
\begin{remark}
It is readily seen that the solution is unique in the small uniformly in $t>0$
with values in $L^{3,\infty}_\sigma(\Omega)$.
The latter part of Theorem \ref{main} tells us that those two solutions must coincide with each other even though
the other one $\widetilde w(t)$ is large. 
This was already found by Takahashi \cite[Theorem 2.3]{Ta22} even under less conditions, in which
the temporal continuity does not play any role; 
to be precise, given $r\in (3,\infty)$, 
there is at most one solution, not necessarily the solution in the sense of Definition \ref{def-sol},
to the weak form (as in Yamazaki \cite{Y}) of \eqref{IE} within the class
\[
w\in L^\infty_{\rm loc}([0,\infty); L^{3,\infty}_\sigma(\Omega))\cap L^\infty_{\rm loc}(0,\infty; L^r_\sigma(\Omega))
\]
subject to $\lim_{t\to 0}\|w(t)\|_{3,\infty}=0$, where the case $v_0=0$ (starting problem) is discussed in \cite{Ta22}.
The idea of the proof is traced back to Fujita and Kato \cite{FuK64}.
\label{rem-5}
\end{remark}
\begin{remark}
If, in particular, $\omega_0=0$, then we have even the $L^3$-decay of $w(t)$ with definite rate in Theorem \ref{main}
under the smallness condition on $|\gamma_1|^{1/2}+|\gamma_2|^{1/2}$
thanks to the wake structure of $v_2\in L^q(\Omega)$ with $q>2$ as in \cite[Theorem 2.1]{Ta21}.
This observation
has the following worth noticing consequence. 
Suppose we have a finite number of
 steady-states $v_i(x)$, $i=1,\ldots N$, $N\ge 3$,
each one characterized by (not all identically zero)
 translational velocities $\gamma_i$, $i=1\ldots N$, and zero angular velocity.
Then, under the assumptions that all $\gamma_i$, $i=1\ldots N$, are below a certain constant, Theorem \ref{main}
refined as above
ensures that we can
describe the
transition between all the states $v_i$ in the
 following sense. We start with the transition $v_1(x)\to v_2(x)$. Let us denote by $\hat{v}_1(x,t)$ the unique solution to \Eqref{2}--\Eqref{3} of the type \Eqref{vsol} with $w(x,t)$ as in the above theorem. Thus, 
$\hat{v}_1(\cdot,t)\in L^3(\Omega)$
(better summability than obtained in Theorem \ref{main}),
$t\in (0,\infty)$, and  for sufficiently large $t$ 
it approaches
$v_2$ in $L^3$.
So, there is $t_1$ such that $\hat{v}_1(x,t_1)$ is ``almost"
 $v_2(x)$, within the precision we want. Then,  we can use $\hat{v}_1(\cdot,t_1)$ as initial datum and, again by Theorem \ref{main}
refined as above,
obtain another solution $\hat{v}_2(x,t)$ that converges to $v_3(x)$ as $t\to\infty$. We may then
continue the procedure until we end up with a solution
 $\hat{v}_{N-1}(x,t)$ with initial datum $\hat{v}_{N-2}(x,t_{N-2})$ and converging to $v_N(x)$ as $t\to\infty$.
\label{rem-6}
\end{remark}

\section{Preparatory results}
\label{pre}

$L^q$-$L^r$ decay estimates and their variants of
the evolution operator $T(t,s)$ together with those of the adjoint $T(t,s)^*$ are developed well, see \cite{Hi18, Hi20}
and \cite{GH21} by the authors of the present paper.
Here, there is no need to give all of them.
The only estimates for later use are collected in the following two propositions.
Some of them are not found in \cite{Hi18, Hi20} but follow from estimates there by interpolation quite easily.
The idea to deduce \eqref{adj-int} is due to Yamazaki \cite{Y}.
In view of \eqref{rigid}, the condition \eqref{quan} below implies that
\[
\sup_{t\geq 0}\big(|\gamma(t)|+|\omega(t)|\big)+\sup_{t>s\geq 0}\frac{|\gamma(t)-\gamma(s)|+|\omega(t)-\omega(s)|}{t-s}\leq Cm,
\]
which determines constants of several estimates of the evolution operator.
We note that $m$ can be large to establish the linear theory.
\begin{proposition}
[\cite{Hi18, Hi20}]
Let $m\in (0,\infty)$ and assume
\begin{equation}
(1+|\psi^\prime|_0)D\leq m,
\label{quan}
\end{equation}
where $|\psi^\prime|_0$ and $D$ are respectively given by \eqref{size-psi} and \eqref{D}.

\begin{enumerate}
\item
Let $1<q<r\leq\infty$.
Then there is a constant $C=C(m,q,r,\Omega)>0$ such that
\begin{equation}
\|T(t,s)f\|_{(q,\infty)}+(t-s)^{(3/q-3/r)/2}\|T(t,s)f\|_r
\leq C\|f\|_{(q,\infty)}
\label{evo-est1}
\end{equation}
for all $t>s\geq 0$ and $f\in L^{q,\infty}_\sigma(\Omega)$.
If, in particular, $1<q<r<\infty$, then 
\begin{equation}
\lim_{h\to 0}\|T(t+h,s)f-T(t,s)f\|_r=0
\label{evo-conti}
\end{equation}
for all $t>s\geq 0$ and $f\in L^{q,\infty}_\sigma(\Omega)$.

Let $1<q<r\leq 3$.
Then there is a constant $C=C(m,q,r,\Omega)>0$ such that
\begin{equation}
\|\nabla T(t,s)f\|_r\leq C(t-s)^{-(3/q-3/r)/2-1/2}\|f\|_{(q,\infty)}
\label{evo-est11}
\end{equation}
for all $t>s\geq 0$ and $f\in L^{q,\infty}_\sigma(\Omega)$.

Given $\varepsilon >0$ arbitrarily, there is a constant $C=C(\varepsilon,m,\Omega)>0$ such that
\begin{equation}
\|\nabla T(t,s)f\|_{(3,\infty)}\leq C(t-s)^{-1/2}(1+t-s)^\varepsilon \|f\|_{(3,\infty)}
\label{evo-est12}
\end{equation}
for all $t>s\geq 0$ and $f\in L^{3,\infty}_\sigma(\Omega)$.

\item
Let $1<q\leq r<\infty$.
Then there is a constant $C=C(m,q,r,\Omega)>0$ such that
\begin{equation}
\|T(t,s)^*g\|_{(r,1)}\leq C(t-s)^{-(3/q-3/r)/2}\|g\|_{(q,1)}
\label{evo-est2}
\end{equation}
for all $t>s\geq 0$ and $g\in L^{q,1}_\sigma(\Omega)$.

Let $1<q\leq r\leq 3$.
Then there is a constant $C=C(m,q,r,\Omega)>0$ such that
\begin{equation}
\|\nabla T(t,s)^*g\|_{(r,1)}\leq C(t-s)^{-(3/q-3/r)/2-1/2}\|g\|_{(q,1)}
\label{evo-est3}
\end{equation}
for all $t>s\geq 0$ and $g\in L^{q,1}_\sigma(\Omega)$.
If, in particular, $1/q-1/r=1/3$ as well as $1<q<r\leq 3$, then there is a constant $C=C(m,q,\Omega)>0$ such that
\begin{equation}
\int_0^t\|\nabla T(t,s)^*g\|_{(r,1)}\,ds\leq C\|g\|_{(q,1)}
\label{adj-int}
\end{equation}
for all $t>0$ and $g\in L^{q,1}_\sigma(\Omega)$.
\end{enumerate}
\label{prop-evo1}
\end{proposition}

{\color{black}
\begin{proof}
The only assertion that is not obvious would be \eqref{evo-est12},
in which we observe a slight loss of the rate of decay for $(t-s)\to\infty$, whereas the smoothing rate near $t=s$ is optimal.
This is, however, a direct consequence of \cite[Proposition 2.2, Theorem 2.1, Remark 2.1]{Hi20} as follows.
In fact, we know
\begin{equation}
\|\nabla T(t,s)f\|_r\leq C(t-s)^{-1/2}(1+t-s)^{\max\{0, (1-3/r)/2\}}\|f\|_r
\label{evo-est13}
\end{equation}
for all $t>s\geq 0$ and $f\in L^r_\sigma(\Omega)$ with $r\in (1,\infty)$, 
meaning that the rate of decay is given by 
$(t-s)^{-3/2r}$ when $r>3$.
This rate is less than usual, however, optimal when $\gamma=\omega=0$. 
With \eqref{evo-est13} at hand, we deduce \eqref{evo-est12} by interpolation 
when taking $q_1>3$ in \eqref{inter} as close to $q=3$
as we wish.
The proof is complete.
\end{proof}
}
\begin{proposition}
[{\cite[Proposition 3.3]{GH21}}]
Let $m\in (0,\infty)$ and assume \eqref{quan}.
Let $\frac{3}{2}<q<r\leq \infty$.
Then there is a constant $C=C(m,q,r,\Omega)>0$ such that the composite operator
$T(t,s)P\mbox{\rm div}$ extends to a bounded operator from $L^{q,\infty}(\Omega)$ 
to $L^r_\sigma(\Omega)$, $r<\infty$, and
to $L^\infty(\Omega)$
subject to
\begin{equation}
\|T(t,s)P\mbox{\rm div $F$}\|_r\leq C(t-s)^{-(3/q-3/r)/2-1/2}\|F\|_{(q,\infty)}
\label{compo-2}
\end{equation}
for all $t>s\geq 0$ and $F\in L^{q,\infty}(\Omega)$. 
If, in particular, $\frac{3}{2}<q<r<\infty$, then 
\begin{equation}
\lim_{h\to 0}\|T(t+h,s)P\mbox{\rm div $F$}-T(t,s)P\mbox{\rm div $F$}\|_r=0
\label{compo-conti}
\end{equation}
for all $t>s\geq 0$ and $F\in L^{q,\infty}(\Omega)$.

Let $\frac{3}{2}<q<r\leq 3$.
Then there is a constant $C=C(m,q,r,\Omega)>0$ such that
\begin{equation}
\|\nabla T(t,s)P\mbox{\rm div $F$}\|_r
\leq C(t-s)^{-(3/q-3/r)/2-1}\|F\|_{(q,\infty)}
\label{compo-3}
\end{equation}
for all $t>s\geq 0$ and $F\in L^{(q,\infty)}(\Omega)$.
\label{prop-evo2}
\end{proposition}

\begin{proof}
Since \eqref{compo-2} was deduced in \cite{GH21}, we verify only \eqref{compo-conti}--\eqref{compo-3}.
Let $0<|h|<\frac{t-s}{2}$, then we have $\frac{t+s}{2}<t-|h|$.
Since $T(\frac{t+s}{2},s)P\mbox{div $F$}\in L^r_\sigma(\Omega)$ for all $F\in L^{q,\infty}(\Omega)$, 
it follows from the strong continuity of $T(t,s)$ on $L^r_\sigma(\Omega)$ that
\begin{equation*}
\begin{split}
&\quad \|T(t+h,s)P\mbox{div $F$}-T(t,s)P\mbox{div $F$}\|_r  \\
&=\|\big\{T(t+h,(t+s)/2)-T(t,(t+s)/2)\big\}T((t+s)/2,s)P\mbox{div $F$}\|_r\to 0
\end{split}
\end{equation*}
as $h\to 0$,
yielding \eqref{compo-conti}.
We use the semigroup property again to find \eqref{compo-3} from \eqref{evo-est13} with $r\leq 3$ and \eqref{compo-2}.
The proof is complete.
\end{proof}

We show the estimate of the force $f$ given by
\Eqref{7}, 
in which
we need an idea to deal with the term $(\omega_0\times x)\cdot\nabla v_2$ since the pointwise decay $\nabla v_2(x)=O(|x|^{-2})$
is not available unlike \cite{GH21, Ta22} unless $\gamma_2\cdot\omega_0=0$
as well as $\omega_0\neq 0$; see Remark \ref{rem-force}.
The smallness condition \eqref{small3} below is required to establish the existence of a unique steady state $v_2$
with desired properties.
Note that the condition \eqref{small3} implies not only \eqref{data-2} via the following lemma but also \eqref{data-3}--\eqref{data-4}.
\begin{lemma}
There exists a constant $\delta^\prime=\delta^\prime(\Omega)>0$ such that if
\begin{equation}
|\gamma_2|+|\omega_0|\leq \delta^\prime,
\label{small3}
\end{equation}
then the function $f$ defined in \Eqref{7} satisfies
\[
f\in L^\infty(0,\infty; L^{3,\infty}(\Omega)). 
\]
Moreover, there exists a constant $C_0=C_0(\Omega)>0$ such that
\[
\sup_{t\geq 0}
\|f(t)\|_{(3,\infty)}\leq C_0(1+|\psi^\prime|_0)D,
\]
where $|\psi^\prime|_0$ and $D$ are respectively given by \eqref{size-psi} and \eqref{D}. 
\label{lem-force}
\end{lemma}
\begin{proof}
With the origin of coordinates in the interior of $\mathscr B$,  set 
$$
\n{u}_1:=\sup_{x\in\Omega}\left(|x|\,|u(x)|\right)\,.
$$
We begin to recall that, under the stated assumptions, from \cite[Theorem 1 and Remark 2]{GS07} it follows that problem \Eqref{1} has a solution $(v_2,p_2)$ such that
$$
v_2\in W^{1,2}_{\rm loc}(\overline{\Omega})\cap D^{2,2}(\Omega)\cap D^{1,2}(\Omega)\,, \ \n{v_2}_1<\infty\,;\ \ p_2\in W^{1,2}(\Omega)\,.
$$ 
Furthermore, this solution satisfies
\be
\|\nabla v_2\|_{1,2}+\|p_2\|_{1,2}+\n{v_2}_1\le c_0\,(|\gamma_2|+|\omega_0|)
=:D_0\,,
\eeq{8}
where $c_0=c_0(\Omega)>0$.
Let $\varphi=\varphi(|x|)$ be a smooth cut-off function that is 0 in a neighborhood of $\partial\Omega$ and 1 for 
$|x|>R>2$.
Setting
\be
w:=\varphi\,v_2\,,\ \ \pi:=\varphi\, p_2\,,
\eeq{9}
from \Eqref{1} we deduce that
\be\left.\ba{ll}\medskip
\Delta w +(\gamma_2+\omega_0\times x)\cdot\nabla w-\omega_0\times w=\nabla \pi+ H\\
\Div w=h\ea\right\}\ \ \mbox{in $\real^3$}\,,
\eeq{10}
where
$$
H:=\varphi\, v_2\cdot\nabla v_2+(\gamma_2\cdot\nabla\varphi
+\Delta\varphi)v_2+2\nabla\varphi\cdot\nabla v_2-p_2\nabla\varphi\,,\ \ h:=\nabla\varphi\cdot v_2\,,
$$
and we used the property $\omega_0\times x\cdot\nabla\varphi=0$. Let $ z$ be a solution to the problem
\be\ba{ll}\medskip
\Div z=h\ \ \mbox{in $\Omega_R:=\{x\in\Omega:\,|x|<R\}$}\,,\\  z\in W^{2,q}_0(\Omega)\,, \ q\in (1,6]\,,\ \ 
\|z\|_{2,q}\le c\,\|h\|_{1,q}\,.
\ea
\eeq{11}
From \Eqref{8}, the properties of $\varphi$ and elementary embedding theorems it follows that $h\in W^{1,q}_0(\Omega_R)$, and that
\be
\|h\|_{1,q}\le c\,D_0
. 
\eeq{12}
Moreover,
$$
\int_{\Omega_R}h=\int_{\Omega_R}\Div(\varphi\, v_2)=\int_{\{|x|=R\}} v_2\cdot\nu=0\,,
$$
so that problem \Eqref{11} is solvable with a constant $c=c(\Omega,q)>0$; see, e.g., \cite[Theorem III.3.3]{G-b}.
Thus, setting
\be
{\sf w}:= w-z\,,
\eeq{13}
from \Eqref{10} and \Eqref{11} we infer
\be\left.\ba{ll}\medskip
\Delta{\sf w} +(\gamma_2+\omega_0\times x)\cdot\nabla{\sf w}-\omega_0\times{\sf w}=\nabla \pi+ G\\
\Div{\sf w}=0\ea\right\}\ \ \mbox{in $\real^3$}\,,
\eeq{14}
where
$$
G:=H-\Delta z -(\gamma_2+\omega_0\times x)\cdot\nabla z+\omega_0\times z\,,
$$
We next observe that, by \Eqref{8},  we have 
\be
\|v_2\cdot\nabla v_2\|_q\le c\,D_0\left(\int_\Omega|x|^{-\frac{6q}{6-q}}\right)^{\frac{6-q}{6q}}
\|\nabla v_2\|_6\le c_1\,D_0\,,\ \ \mbox{for all $q\in (2,6)$}\,,
\eeq{14_0}
with $c_1=c_1(\Omega,q)>0$. Therefore, again by \Eqref{8},  embedding theorems, \Eqref{11}$_3$ and \Eqref{12} we show
\be
\|G\|_{\color{black}q,\mathbb R^3}\le c\,D_0\,,\ \ \mbox{for all $q\in(2,6)$}\,.
\eeq{15}
As a consequence,  by \cite[Theorem 1.2]{GK13} and \Eqref{8}, it follows, in particular, that
$$
\|D^2{\sf w}\|_{q,\mathbb R^3}+\|\nabla \pi\|_{q,\mathbb R^3}\le c\,D_0\,,\ \ \mbox{for all $q\in(2,6)$}\,.
$$
Combining the latter with \Eqref{11}$_3$--\Eqref{13}, \Eqref{14}$_1$
and \Eqref{15} furnishes 
$$
\|(\gamma_2+\omega_0\times x)\cdot\nabla w-\omega_0\times w\|_{q}\le c\,D_0\,, \ \ \mbox{for all $q\in (2,6)$}\,.
$$
Because of \Eqref{9} and the properties of $\varphi$, this implies
$$
\|(\gamma_2+\omega_0\times x)\cdot\nabla v_2-\omega_0\times v_2\|_{q,\{|x|>R\}}\le c\,D_0\,, \ \ \mbox{for all $q\in (2,6)$}\,.
$$
However, by \Eqref{8}, we also have
$$
\|(\gamma_2+\omega_0\times x)\cdot\nabla v_2-\omega_0\times v_2\|_{q,\{x\in\Omega:|x|<R\}}\le c\,D_0\,, \ \ \mbox{for all $q\in (2,6)$}
$$
so that we conclude
\be
\|(\gamma_2+\omega_0\times x)\cdot\nabla v_2-\omega_0\times v_2\|_{q}\le c\,D_0\,, \ \ \mbox{for all $q\in (2,6)$}\,.
\eeq{16}
Furthermore,  by \Eqref{8},  embedding, and the property of $\psi$ 
\be
\|\gamma_1\cdot\nabla v_2\|_q\le c\, |\gamma_1| D_0\,, \ \ \mbox{for all $q\in (2,6)$}\,.
\eeq{17}
Finally, again by \Eqref{8}, we get
\be
\| v_2\|_{(3,\infty)}\le c\,D_0\,.
\eeq{18}
Thus, collecting 
\Eqref{14_0}, \Eqref{16}--\Eqref{18} and taking again into account the property of $\psi$, we complete the proof of the  lemma. 
\end{proof}
\begin{remark}
Let us consider the general case in which the initial state $v_1$ is generated by the angular velocity $\omega_1$ as well as
the translation $\gamma_1$.
Then the force $f$ in \Eqref{7} involves
\[
(\psi-\psi^2)\big[(\omega_1\times x)\cdot\nabla v_2-\omega_1\times v_2\big]
\]
as well.
The problematic term is, in fact, $(\omega_1\times x)\cdot \nabla v_2$, for which the argument in the proof of Lemma \ref{lem-force}
cannot work because $v_2$ is
not, in general, related to $\omega_1$. However, there are special cases where this term can be handled. First, when $\omega_1=\kappa\, \omega_0$, for some $\kappa\in \real$. Second,
if $\omega_1\neq \kappa\,\omega_0$, but $\omega_0 \neq 0$ and $\gamma_2\cdot\omega_0=0$,
then one can fortunately deduce 
$(\omega_1\times x)\cdot\nabla v_2\in L^{3,\infty}(\Omega)$ from the pointwise decay $\nabla v_2(x)=O(|x|^{-2})$;
in contrast with {this case},
if $v_2$ possesses the wake structure, the transition problem is still out of reach.
\label{rem-force}
\end{remark}

\section{Proof of Theorem \ref{main}}
\label{proof}

In this section we will show the main theorem.
Let $q\in [3,\infty)$.
For a strongly measurable function $w: (0,\infty)\to L^{q,\infty}_\sigma(\Omega)$ and $t\in (0,\infty)$, we set
\begin{equation}
[w]_{q,t}:=\mbox{esssup}_{s\in (0,t)}\,s^{1/2-3/2q}\|w(s)\|_{(q,\infty)}, \qquad
[w]_q:=\sup_{t>0}\,[w]_{q,t}.
\label{beha}
\end{equation}
The following lemma provides us with several properties of
\begin{equation}
(\Psi w)(t):=\int_0^tT(t,s)P\mbox{div $(Fw)$}(s)\,ds.
\label{Psi}
\end{equation}
\begin{lemma}
Let \eqref{small3} be satisfied.
Let $m\in (0,\infty)$ and assume \eqref{quan}.
Suppose that
\[
w\in L^\infty_{\rm loc}(0,\infty;
L^{3,\infty}_\sigma(\Omega)\cap L^{q,\infty}_\sigma(\Omega))
\]
with some $q\in (3,\infty)$ and that
\[
[w]_{3,t}+[w]_{q,t}<\infty
\]
for every $t>0$.
Then we have the following.

\begin{enumerate}
\item
The integral \eqref{Psi} is  Bochner integrable in $L^3_\sigma(\Omega)$ and in $L^q_\sigma(\Omega)$, so that
the function $(\Psi w)(t)$ is well-defined subject to
\begin{equation}
\sup_{s\in (0,t)}\|(\Psi w)(s)\|_3\leq
CD[w]_{q,t}+C[w]_{3,t}[w]_{q,t}
\label{3-est}
\end{equation}
for all $t>0$ with some $C=C(m,q,\Omega)>0$ independent of $t$, where $D$ is given by \eqref{D}. 
Moreover, we have
\begin{equation}
\Psi w\in C((0,\infty); L^3_\sigma(\Omega)).
\label{Psi-conti1}
\end{equation}

\item
We have
\begin{equation}
[\Psi w]_{q,t}\leq
CD[w]_{q,t}+C[w]_{3,t}[w]_{q,t}
\label{q-est}
\end{equation}
for all $t>0$ with some $C=C(m,q,\Omega)>0$ independent of $t$.

\item
If, in particular, $q\in (6,\infty)$, then the integral \eqref{Psi} is the Bochner integrable in $L^\infty(\Omega)$ as well as in
$L^r_\sigma(\Omega)$ for every $r\in (3,\infty)$.
Moreover, we have
\begin{equation}
\Psi w\in C((0,\infty); L^r_\sigma(\Omega)), \qquad 
\Psi w\in C_{w^*}((0,\infty); L^\infty(\Omega))
\label{Psi-conti2}
\end{equation}
for every $r\in (3,\infty)$.

\item
If, in particular, 
\[
[w]_3+[w]_q<\infty
\]
for some $q\in (6,\infty)$, then
\begin{equation}
\|(\Psi w)(t)\|_\infty=O(t^{-1/2+3/2q}),
\label{decay-inf}
\end{equation}
\begin{equation}
\|(\Psi w)(t)\|_q=O(t^{-1/2+3/2q}\log t),
\label{decay-q}
\end{equation}
as $t\to\infty$.

\item
In addition to the assumption of the 
item 4, suppose that
\[
\nabla w\in L^\infty_{\rm loc}(0,\infty; L^{3,\infty}(\Omega))
\]
subject to
\begin{equation}
\{\nabla w\}_\varepsilon:=\mbox{\rm esssup}_{t>0}\, t^{1/2}(1+t)^{-\varepsilon} \|\nabla w(t)\|_{(3,\infty)}<\infty
\label{grad-ep}
\end{equation}
for some $\varepsilon>0$.
Then there is a constant $C=C(m,q,\Omega)>0$ independent of $t$ such that
\begin{equation}
\nabla (\Psi w)\in C_w((0,\infty); L^3(\Omega)),
\label{grad-conti}
\end{equation}
\begin{equation}
\|\nabla (\Psi w)(t)\|_3\leq C\Big[D\big([w]_q+\{\nabla w\}_\varepsilon\big)+[w]_q\{\nabla w\}_\varepsilon\Big]
t^{-1/2}(1+t)^{\max\{\varepsilon,\,3/2q\}}, 
\label{grad-est}
\end{equation}
for all $t>0$, 
where $D$ is given by \eqref{D}.
\end{enumerate}
\label{lem-Psi}
\end{lemma}

\begin{proof}
As described just before Lemma \ref{lem-force}, the condition \eqref{small3} allows us to use \eqref{data-3}
{\color{black}
--\eqref{data-4}.
}
It follows from \eqref{compo-2} and the weak H\"older inequality that
\[
\|T(t,s)P\mbox{div $(Fw)$}(s)\|_p\leq C(t-s)^{-3/2q-1/2}\|w\|_{(q,\infty)}\big(\|w\|_{(p,\infty)}+\|b+\psi v_2\|_{(p,\infty)}\big)
\]
for each $p\in \{3,q\}$, yielding the Bochner integrability in both spaces $L^3_\sigma(\Omega)$ and $L^q_\sigma(\Omega)$.
By \eqref{data-3}
the case $p=3$ immediately leads to \eqref{3-est}, while estimate for the other case $p=q$ yields 
\begin{equation}
\|(\Psi w)(t)\|_q
\leq C[w]_{q,t}\big([w]_{q,t}\,t^{-1/2+3/2q}+D\big)
\label{Psi-q}
\end{equation}
for all $t>0$.

Since \eqref{Psi-q} itself is not useful, let us show \eqref{q-est} instead.
To this end, we make use of the Lorentz space
as in \cite{GH21, HS09, Ta22}, where the idea is due to \cite{Y}.
Consider
\begin{equation*}
\begin{split}
\langle (\Psi w)(t),\phi\rangle
&=\int_0^{t/2}\langle (T(t,s)P\mbox{div $(Fw)$}(s),\phi\rangle\,ds
-\int_{t/2}^t\langle (Fw)(s), \nabla T(t,s)^*\phi\rangle\,ds  \\
&=:I+II
\end{split}
\end{equation*}
for $\phi\in L^{q^\prime,1}_\sigma(\Omega)$.
We apply \eqref{adj-int} to
\begin{equation*}
\begin{split}
II&\leq \int_{t/2}^t\|w\|_{(q,\infty)}\big(\|w\|_{(3,\infty)}+\|b+\psi v_2\|_{(3,\infty)}\big)\|\nabla T(t,s)^*\phi\|_{(r,1)}\,ds  \\
&\leq Ct^{-1/2+3/2q}[w]_{q,t}\big([w]_{3,t}+D\big)\int_{t/2}^t\|\nabla T(t,s)^*\phi\|_{(r,1)}\,ds,
\end{split}
\end{equation*}
where $r\in (\frac{3}{2},3)$ is determined by $1/q+1/r=2/3$, whereas we still use \eqref{compo-2} to infer
\[
I\leq C\int_0^{t/2}(t-s)^{-1}\|w\|_{(q,\infty)}\big(\|w\|_{(3,\infty)}+\|b+\psi v_2\|_{(3,\infty)}\big)\|\phi\|_{(q^\prime,1)}\,ds,
\]
where $1/q^\prime+1/q=1$.
In this way, we are led to \eqref{q-est} by duality.

Using \eqref{compo-conti}, we are able to show \eqref{Psi-conti1} in the same way as in \cite[Lemma 4.6]{Ta22}.
Since the right continuity is easier, let us discuss merely the left continuity.
We fix $t>0$, take $h>0$ small enough and 
consider
\begin{equation*}
\begin{split}
&\quad (\Psi w)(t-h)-(\Psi w)(t)  \\
&=\int_0^{t-h}\Big(T(t-h,s)P\mbox{div $(Fw)$}(s)-T(t,s)P\mbox{div $(Fw)$}(s)\Big)\,ds \\
&\quad +\int_{t-h}^t T(t,s)P\mbox{div $(Fw)$}(s)\,ds=:III+IV.
\end{split}
\end{equation*}
Given $\varepsilon>0$ arbitrarily, we choose $\delta\in (0,\frac{t}{2})$
such that
$\int_0^\delta \tau^{-3/2q-1/2}\,d\tau\leq\varepsilon$ and then split $III$ into
\begin{equation}
\begin{split}
\|III\|_3
&=\left\|\int_0^{t-\delta}+\int_{t-\delta}^{t-h}\right\|_3  \\
&\leq \int_0^{t-\delta}\|T(t-h,s)P\mbox{div $(Fw)$}(s)-T(t,s)P\mbox{div $(Fw)$}(s)\|_3\,ds  \\
&\quad +C\varepsilon\,[w]_{q,t}\big([w]_{3,t}+D\big)t^{-1/2+3/2q},
\end{split}
\label{3-conti}
\end{equation}
where $0<h<\frac{\delta}{2}<\frac{t}{4}$.
Since the first term goes to zero as $h\to 0$ by virtue of \eqref{compo-conti} and by the Lebesgue convergence theorem,
we furnish $\lim_{h\to 0}\|III\|_3=0$.
It is readily seen from \eqref{compo-2} that $IV$ is even H\"older continuous.

Assume $q\in (6,\infty)$, then we utilize \eqref{compo-2} with $r=\infty$ to find
\[
\|T(t,s)P\mbox{div $(Fw)$}(s)\|_\infty
\leq C(t-s)^{-3/q-1/2}\|w\|_{(q,\infty)}\big(\|w\|_{(q,\infty)}+\|b+\psi v_2\|_{(q,\infty)}\big)
\]
that implies the Bochner integrability in $L^\infty(\Omega)$ 
along with
\begin{equation}
\|(\Psi w)(t)\|_\infty\leq C[w]_{q,t}\big([w]_{q,t}\,t^{-1/2}+Dt^{-3/2q}\big)
\label{Psi-infty}
\end{equation}
for all $t>0$ and, thereby, it is the Bochner integrable in $L^r_\sigma(\Omega)$ with $r\in (3,\infty)$ as well
since so is in $L^3_\sigma(\Omega)$.
The proof of \eqref{Psi-conti2} for $r<\infty$ is essentially the same as the one for \eqref{Psi-conti1},
where the right-hand side of \eqref{3-conti} is replaced by
\[
\int_0^{t-\delta}\|\cdots\|_r\,ds+C\varepsilon\, [w]_{q,t}\big([w]_{q,t}\,t^{-1+3/q}+Dt^{-1/2+3/2q}\big)
\]
with a slight change of choice of $\delta$ accordingly
for the proof of the left continuity.
As to the remaining case $r=\infty$, it suffices to show
\begin{equation}
\lim_{h\to 0}\langle (\Psi w)(t+h)-(\Psi w)(t),\phi\rangle=0
\label{weak-conti}
\end{equation}
for every $\phi\in C_0^\infty(\Omega)$ since we have \eqref{Psi-infty}.
To consider the right continuity, let $h>0$ and let $r\in (\frac{3}{2},2)$ fulfill $1/q+1/r=2/3$.
Then, by the backward semigroup property of $T(t,s)^*$, \eqref{evo-est3} and \eqref{compo-2},
we have
\begin{equation*}
\begin{split}
&\quad |\langle (\Psi w)(t+h)-(\Psi w)(t),\phi\rangle|  \\
&\leq \int_0^t|\langle (Fw)(s), \nabla \big\{T(t+h,s)^*-T(t,s)^*\big\}P\phi\rangle|\,ds \\
&\quad +\int_t^{t+h}|\langle T(t+h,s)P\mbox{div $(Fw)$}(s),\phi\rangle|\,ds  \\
&\leq C[w]_{q,t}\big([w]_{3,t}+D\big)\int_0^ts^{-1/2+3/2q}
\|\nabla T(t,s)^*\big\{T(t+h,t)^*P\phi-P\phi\big\}\|_{(r,1)}\,ds  \\
&\quad +C[w]_{q,t+1}\big([w]_{3,t+1}+D\big)\int_t^{t+h}(t+h-s)^{-1/2-3/2q}s^{-1/2+3/2q}\|P\phi\|_{(3/2,1)}\,ds \\
&\leq C[w]_{q,t+1}\big([w]_{3,t+1}+D\big)
\Big(\|T(t+h,t)^*P\phi-P\phi\|_{(3/2,1)}+t^{-1/2+3/2q}h^{1/2-3/2q}\|P\phi\|_{(3/2,1)}\Big)
\end{split}
\end{equation*}
which goes to zero as $h\to 0$ for every $\phi\in C_0^\infty(\Omega)$.
Since the left continuity is discussed in an analogous way, we are led to \eqref{weak-conti}.

We next
show the large time behavior \eqref{decay-inf}--\eqref{decay-q} by following \cite{GH21}.
We already know from the first term of the right-hand side of \eqref{Psi-q} and \eqref{Psi-infty}, respectively, that
\[
\left\|\int_0^t T(t,s)P\mbox{div}(w\otimes w)(s)\,ds\right\|_r
\leq C[w]_{q}^2\,t^{-1/2+3/2r}, \qquad r\in \{q,\infty\},
\]
provided $[w]_q<\infty$.
We split the other part of $(\Psi w)(t)$ into
\[
\int_0^t T(t,s)P\mbox{div}\big[w\otimes (b+\psi v_2)+(b+\psi v_2)\otimes w\big](s)\,ds
=\int_0^{t-1}+\int_{t-1}^t=:V+VI
\]
for $t>2$ and make use of \eqref{compo-2} with $r=\infty$ to estimate each of them.
Then we have
\begin{equation}
\begin{split}
\|V\|_\infty
&\leq C\int_0^{t-1}(t-s)^{-1-3/2q}\|w\|_{(q,\infty)}
\|b+\psi v_2\|_{(3,\infty)}\,ds  \\
&\leq CD[w]_q
\left(t^{-1-3/2q}\int_0^{t/2}s^{-1/2+3/2q}\,ds+t^{-1/2+3/2q}\int_{t/2}^{t-1}(t-s)^{-1-3/2q}\,ds\right)  \\
&\leq CD[w]_q
\big(t^{-1/2}+t^{-1/2+3/2q}\big),
\end{split}
\label{decay-key}
\end{equation}
while
\begin{equation*}
\begin{split}
\|VI\|_\infty
&\leq C\int_{t-1}^t(t-s)^{-3/q-1/2}\|w\|_{(q,\infty)}
\|b+\psi v_2\|_{(q,\infty)}\,ds  \\
&\leq CD[w]_q\,t^{-1/2+3/2q}.
\end{split}
\end{equation*}
Analogous computations yield
\begin{equation*}
\begin{split}
\|V\|_q
&\leq C\int_0^{t-1}(t-s)^{-1}\|w\|_{(q,\infty)}
\|b+\psi v_2\|_{(3,\infty)}\,ds  \\
&\leq CD[w]_q\,t^{-1/2+3/2q}(1+\log t)
\end{split}
\end{equation*}
and
\begin{equation*}
\|VI\|_q\leq CD[w]_q\,t^{-1/2+3/2q},
\end{equation*}
which imply \eqref{decay-q}.

Finally, let us discuss $\nabla (\Psi w)(t)$.
The continuity \eqref{Psi-conti1} of $\Psi w$ immediately implies that
\[
\langle \nabla (\Psi w)(t+h)-\nabla (\Psi w)(t), \varphi\rangle
=-\langle (\Psi w)(t+h)-(\Psi w)(t), P\mbox{div $\varphi$}\rangle\to 0\qquad (h\to 0)
\]
for all $\varphi\in C_0^\infty(\Omega)^{3\times 3}$; thus, once we have the $L^3$-estimate \eqref{grad-est},
we are led to \eqref{grad-conti}.
To show \eqref{grad-est}, we have to derive estimate of each of three terms
\[
(\Psi w)(t)=(\Psi_1w)(t)+(\Psi_2w)(t)+(\Psi_3w)(t), \qquad
(\Psi_kw)(t):=\int_0^tT(t,s)P\mbox{div $(F_kw)$}(s)\,ds,
\]
where
\[
F_1w=w\otimes w, \qquad F_2w=w\otimes (b+\psi v_2), \qquad F_3w=(b+\psi v_2)\otimes w.
\]
We use \eqref{evo-est11} to get
\begin{equation}
\|\nabla(\Psi_1w)(t)\|_3\leq C\int_0^t(t-s)^{-1/2-3/2q}\|w\|_{(q,\infty)}\|\nabla w\|_{(3,\infty)}\,ds
\leq C[w]_q\{\nabla w\}_\varepsilon\, t^{-1/2}(1+t)^\varepsilon
\label{grad-1}
\end{equation}
for all $t>0$.
We fix $p\in (3,\infty)$ and recall \eqref{data-3}--\eqref{data-4} to infer
\begin{equation}
\|\nabla(\Psi_2w)(t)\|_3
\leq C\int_0^t(t-s)^{-1/2-3/2p}\|b+\psi v_2\|_{(p,\infty)}\|\nabla w\|_{(3,\infty)}\,ds
\leq CD\{\nabla w\}_\varepsilon\, t^{-3/2p},
\label{grad-2}
\end{equation}
\begin{equation}
\|\nabla(\Psi_3w)(t)\|_3
\leq C\int_0^t(t-s)^{-1/2-3/2q}\|w\|_{(q,\infty)}\|\nabla(b+\psi v_2)\|_3\,ds
\leq CD[w]_q,
\label{grad-3}
\end{equation}
for all $t\in (0,2]$, where we see less singularity near $t=0$ than \eqref{grad-1}.
Let $t>2$, and split $\nabla(\Psi_kw)(t)$, $k=2,3$, into
\[
\nabla(\Psi_kw)(t)=\int_0^{t-1}+\int_{t-1}^t=:J_{k1}+J_{k2}.
\]
We still deal with $J_{k1}$, $k=2,3$, as the divergence form and employ \eqref{compo-3} to find
\begin{equation}
\|J_{21}\|_3+\|J_{31}\|_3
\leq C\int_0^{t-1}(t-s)^{-1-3/2q}\|w\|_{(q,\infty)}\|b+\psi v_2\|_{(3,\infty)}\,ds
\leq CD[w]_q\big(t^{-1/2}+t^{-1/2+3/2q}\big)
\label{grad-4}
\end{equation}
for all $t>2$, which is exactly the same as in \eqref{decay-key}.
As for $J_{k2}$, $k=2,3$, we estimate them, respectively, as in \eqref{grad-2}--\eqref{grad-3} to obtain 
\begin{equation}
\|J_{22}\|_3
\leq C\int_{t-1}^t(t-s)^{-1/2-3/2p}\|b+\psi v_2\|_{(p,\infty)}\|\nabla w\|_{(3,\infty)}\,ds
\leq CD\{\nabla w\}_\varepsilon\, t^{-1/2+\varepsilon},
\label{grad-5}
\end{equation}
\begin{equation}
\|J_{32}\|_3
\leq C\int_{t-1}^t(t-s)^{-1/2-3/2q}\|w\|_{(q,\infty)}\|\nabla (b+\psi v_2)\|_3\,ds
\leq CD[w]_q\,t^{-1/2+3/2q},
\label{grad-6}
\end{equation}
for all $t>2$.
Collecting \eqref{grad-1}--\eqref{grad-6}, we conclude \eqref{grad-est}.
The proof is complete.
\end{proof}

Set
\begin{equation}
w_0(t):=T(t,0)v_0+w_1(t), \qquad
w_1(t):=\int_0^tT(t,s)Pg(s)\,ds.
\label{top}
\end{equation}
Then, we readily observe the following properties of $w_0(t)$.
\begin{lemma}
Let $\delta^{\prime\prime}=\delta^{\prime\prime}(\Omega)\in (0,\delta^\prime]$ be a constant such that the condition
$|\gamma_1|^{1/2}\leq \delta^{\prime\prime}$ leads to the existence of a unique steady state $v_1$ satisfying
$\|v_1\|_3\leq C|\gamma_1|^{1/2}$,
where $\delta^\prime$ is the constant given in Lemma \ref{lem-force}.
Assume that
$D\leq \delta^{\prime\prime}\leq \delta^\prime$,
which implies all of \eqref{data-1}--\eqref{data-4},
where
$D$ is given by \eqref{D}.
Let $m\in (0,\infty)$ and assume \eqref{quan}.
Then we have
\begin{equation}
w_0\in C_{w^*}((0,\infty); L^{3,\infty}_\sigma(\Omega)), \qquad
\nabla w_0\in C_{w^*}((0,\infty); L^{3,\infty}(\Omega)),
\label{top-conti-1}
\end{equation}
\begin{equation}
w_0\in C((0,\infty); L^q_\sigma(\Omega))\cap C_{w^*}((0,\infty); L^\infty(\Omega)),
\label{top-conti-2}
\end{equation}
\begin{equation}
\lim_{t\to 0}\|w_0(t)-v_0\|_{(3,\infty)}=0, \qquad
\lim_{t\to 0}\,[w_0]_{q,t}=0,
\label{top-ic}
\end{equation}
\begin{equation}
[w_0]_3+\sup_{t>0} t^{1/2-3/2q}\|w_0(t)\|_q
\leq c(1+|\psi^\prime|_0)D,
\label{top-est}
\end{equation}
\begin{equation}
\|w_0(t)\|_\infty=O(t^{-1/2})\qquad (t\to\infty),
\label{top-decay}
\end{equation}
for every $q\in (3,\infty)$
with some $c=c(m,q,\Omega)>0$.
Moreover, given $\varepsilon>0$ arbitrarily, there is a constant $C=C(\varepsilon,m,\Omega)>0$ such that
\begin{equation}
\{\nabla w_0\}_\varepsilon
\leq C(1+|\psi^\prime|_0)D.
\label{top-grad}
\end{equation}
Here, $[\,\cdot\,]_{q,t},\, [\,\cdot\,]_q$ and $\{\nabla(\cdot)\}_\varepsilon$ are defined in \eqref{beha} and \eqref{grad-ep}.
\label{lem-top}
\end{lemma}

\begin{proof}
Let 
$q\in (3,\infty)$ and $r\in \{q,\infty\}$.
By \eqref{data-2}, \eqref{evo-est1}
and \eqref{evo-est12}
we see that
\[
\|w_1(t)\|_{(3,\infty)}\leq C(1+|\psi^\prime|_0)D\min\{1,t\}
\]
and that
\begin{equation}
\begin{split}
\|w_1(t)\|_r
&\leq C\int_0^{\min\{1,t\}}(t-s)^{-1/2+3/2r}\|g(s)\|_{(3,\infty)}\,ds   \\
&\leq C(1+|\psi^\prime|_0)D
\left\{
\begin{array}{ll}
t^{1/2+3/2r}\quad &(t<2), \\
t^{-1/2+3/2r}&(t\geq 2),
\end{array}
\right.
\end{split}
\label{w1-est}
\end{equation}
as well as
\begin{equation*}
\begin{split}
\|\nabla w_1(t)\|_{(3,\infty)}
&\leq C\int_0^{\min\{1,t\}}(t-s)^{-1/2}(1+t-s)^\varepsilon \|g(s)\|_{(3,\infty)}\,ds   \\
&\leq C(1+|\psi^\prime|_0)D
\left\{
\begin{array}{ll}
t^{1/2}\quad &(t<2), \\
t^{-1/2+\varepsilon}&(t\geq 2).
\end{array}
\right.
\end{split}
\end{equation*}
These together with \eqref{data-1}, \eqref{evo-est1},
\eqref{evo-est13} and
\[
\lim_{t\to 0}\|T(t,0)v_0-v_0\|_3=0, \qquad \lim_{t\to 0}t^{1/2-3/2r}\|T(t,0)v_0\|_r=0,
\]
the latter of which is a specific case of 
\cite[(2.21)]{AH}, imply \eqref{top-ic}--\eqref{top-grad}.

Using \eqref{evo-conti}, we see the strong continuity with values in $L^q_\sigma(\Omega)$ described in \eqref{top-conti-2}
along the same manner as in the proof of \eqref{Psi-conti1}.
With the aid of \eqref{evo-est2} in the analogous way to the proof of \eqref{weak-conti},
we also infer
\begin{equation}
|\langle w_0(t+h)-w_0(t), \phi\rangle|
\leq C(1+|\psi^\prime|_0)D\Big((1+t)\|T(t+h,t)^*P\phi-P\phi\|_{(3/2,1)}+h\|P\phi\|_{(3/2,1)}\Big)
\label{weak-conti2}
\end{equation}
for all $\phi\in C_0^\infty(\Omega)^3$
and $0<t<t+h$, where \eqref{data-1} and \eqref{data-2} are used.
Since estimate for the other case $0<t-h<t$ is similar and since we have \eqref{w1-est} with $r=\infty$
(and, therefore, estimate of $\|w_0(t)\|_\infty$),
we are led to the weakly-$*$ continuity with values in $L^\infty(\Omega)$ described in \eqref{top-conti-2}.
Finally, $\mbox{\eqref{top-conti-1}}_1$ follows from
\eqref{weak-conti2} with $\phi\in L^{3/2,1}_\sigma(\Omega)$
and, thereby, we have
\[
\langle \nabla w_0(t+h)-\nabla w_0(t), \varphi\rangle
=-\langle w_0(t+h)-w_0(t), P\mbox{div $\varphi$}\rangle\to 0\qquad (h\to 0)
\]
for all $\varphi\in C_0^\infty(\Omega)^{3\times 3}$, which leads to $\mbox{\eqref{top-conti-1}}_2$ on account of \eqref{top-grad}.
The proof is complete.
\end{proof}
\begin{remark}
It seems to be difficult to show $w_1\in C((0,\infty); L^{3,\infty}_\sigma(\Omega))$ better than \eqref{top-conti-1}
because the evolution operator $T(t,s)$ is not of parabolic type (unless the rotation is absent)
and because $C^\infty_{0,\sigma}(\Omega)$ is not dense in $L^{3,\infty}_\sigma(\Omega)$.
\label{rem-conti}
\end{remark}

Before the proof of the main part of Theorem \ref{main},
let us show the uniqueness of solutions independently of the existence result.
\begin{lemma}
Suppose that $w_1$ and $w_2$ are global solutions to \eqref{IE} in the sense of Definition \ref{def-sol}
and that 
\[
w_1,\, w_2\in L^\infty_{\rm loc}(0,\infty; L^{r,\infty}_\sigma(\Omega))
\]
with some $r\in (3,\infty)$.
There is a constant $\widetilde\delta=\widetilde\delta(\Omega)\in (0,1]$ 
independent of $r\in (3,\infty)$ and $|\psi^\prime|_0$
such that if
\[
D\leq \frac{\widetilde\delta}{1+|\psi^\prime|_0}
\]
as well as $D\leq \delta^{\prime\prime}$,
then $w_1=w_2$,
where $D$ is given by $\eqref{D}$ and $\delta^{\prime\prime}$ is the constant given in Lemma \ref{lem-top}.
\label{lem-uni}
\end{lemma}

\begin{proof}
Since $D\leq \delta^{\prime\prime}$, we have \eqref{data-1} and \eqref{data-3}.
Assume in addition that
$D\leq\frac{1}{1+|\psi^\prime|_0}$, and let us employ Propositions \ref{prop-evo1} and \ref{prop-evo2} for $m=1$, see \eqref{quan}.
We follow the argument of \cite{FuK64, Ta22}.
Set $w=w_1-w_2$, then it obeys
\begin{equation}
w(t)=\int_0^tT(t,s)P\mbox{div $(Gw)$}(s)\,ds
\label{duha-G}
\end{equation}
with
\begin{equation}
Gw=w\otimes w_1+w_2\otimes w+w\otimes (b+\psi v_2)+(b+\psi v_2)\otimes w.
\label{G}
\end{equation}
For every $\phi\in L^{3/2,1}_\sigma(\Omega)$, we use \eqref{adj-int} together with \eqref{data-3} to find
\begin{equation*}
\begin{split}
|\langle w(t),\phi\rangle|
&\leq \int_0^t\|(Gw)(s)\|_{(3/2,\infty)}\|\nabla T(t,s)^*\phi\|_{(3,1)}\,ds  \\
&\leq C[w]_{3,t}\big([w_1]_{3,t}+[w_2]_{3,t}+D\big)\|\phi\|_{(3/2,1)}
\end{split}
\end{equation*}
for all $t>0$.
Given $\varepsilon >0$ arbitrarily, it follows from \eqref{IC} together with \eqref{data-1} that there is $t_0>0$ satisfying
\[
[w_j]_{3,t_0}\leq \|v_0\|_{(3,\infty)}+\varepsilon\leq C|\gamma_1|^{1/2}+\varepsilon, \qquad j\in \{1,2\}.
\]
Thus, there is a constant $\widetilde\delta=\widetilde\delta(\Omega)\in (0,1]$ such that if
$D<\widetilde\delta$, we deduce
$[w]_{3,t_0}=0$ by duality when taking $\varepsilon>0$ small enough. 

Given ${\mathcal T}\in (0,\infty)$ arbitrarily,
we next show the existence of $\tau_*=\tau_*({\mathcal T})>0$ independent of $\tau\in [t_0,{\mathcal T})$ with the following property:
if $w=0$ on $[0,\tau)$, then $w=0$ holds true on $[0,\tau+\tau_*)$ as long as $\tau+\tau_*< {\mathcal T}$
(otherwise, $w=0$ on $[0,{\mathcal T})$).
We then employ this with $\tau=t_0, t_0+\tau_*, t_0+2\tau_*,\cdots$ to accomplish the proof of $w=0$ on $[0,{\mathcal T})$.
Set $M_j:=\mbox{esssup}_{t\in (t_0,{\mathcal T})}\|w_j(t)\|_{(r,\infty)}$ for $j\in \{1,2\}$, then we see from \eqref{data-3} that
\begin{equation*}
\begin{split}
\|w(t)\|_{(3,\infty)} 
&\leq C\int_\tau^t (t-s)^{-1/2-3/2r}\|w\|_{(3,\infty)}\big(\|w_1\|_{(r,\infty)}+\|w_2\|_{(r,\infty)}+\|b+\psi v_2\|_{(r,\infty)}\big)\,ds  \\
&\leq C(M_1+M_2+D)(t-\tau)^{1/2-3/2r}\mbox{esssup}_{s\in (\tau,t)}\|w(s)\|_{(3,\infty)}
\end{split}
\end{equation*}
which gives us the desired $\tau_*$.
The proof is complete.
\end{proof}

We close the paper with the proof of Theorem \ref{main}.

\medskip
\noindent
{\em Proof of Theorem \ref{main}}.
Let $D\leq \delta^{\prime\prime}$, where $\delta^{\prime\prime}$ is the constant given in Lemma \ref{lem-top}.
We assume \eqref{quan} with $m=1$, and make use of 
Lemma \ref{lem-Psi} (with $q$ specified below) and Lemma \ref{lem-top} for $m=1$.
Both conditions follow from \eqref{small} with $\delta\leq\min\{\delta^{\prime\prime},1\}$.
Given $\varepsilon \in (0,\frac{1}{4})$, we set $q=\frac{3}{2\varepsilon}\in (6,\infty)$ and define the function space $X_q$ by
\begin{equation*}
\begin{split}
X_q:=
&\Big\{w\in C_{w^*}((0,\infty); L^{3,\infty}_\sigma(\Omega)\cap L^{q,\infty}_\sigma(\Omega));\;
\nabla w\in C_{w^*}((0,\infty); L^{3,\infty}(\Omega)),\;  \\
&\qquad\qquad\qquad\qquad\qquad
\lim_{t\to 0}\,[w]_{q,t}=0,\;
[w]_3+[w]_q+\{\nabla w\}_\varepsilon <\infty
\Big\}
\end{split}
\end{equation*}
which is a Banach space endowed with norm
$\|w\|_{X_q}:=[w]_3+[w]_q+\{\nabla w\}_\varepsilon$, where $[\,\cdot\,]_{q,t},\, [\,\cdot\,]_q$ and $\{\nabla(\cdot)\}_\varepsilon$
are defined in \eqref{beha} and \eqref{grad-ep}.

From \eqref{3-est}--\eqref{Psi-conti2},
\eqref{grad-conti}--\eqref{grad-est},
\eqref{top-conti-1}--\eqref{top-est}
and \eqref{top-grad}
it follows that $w\in X_q$ implies 
\[
\Phi w:=w_0+\Psi w\in X_q
\]
along with
\begin{equation} 
\|\Phi w\|_{X_q}\leq c_0(1+|\psi^\prime|_0)D+c_1D\|w\|_{X_q}+c_2\|w\|_{X_q}^2.
\label{into}
\end{equation}
Exactly the same computations for \eqref{duha-G}--\eqref{G}
as in the proof of \eqref{3-est}, \eqref{q-est}
and \eqref{grad-est}
lead us to
\[
\|\Phi w_1-\Phi w_2\|_{X_q}\leq \big(c_1D+c_2\|w_1\|_{X_q}+c_2\|w_2\|_{X_q}\big)\|w_1-w_2\|_{X_q}
\]
for all $w_1,\, w_2\in X_q$ with the same constants $c_1$ and $c_2$ as in \eqref{into}.
We thus find that the map $\Phi$ has a fixed point $w$ being in the ball of $X_q$ with radius, say, $2c_0(1+|\psi^\prime|_0)D$
provided that $(1+|\psi^\prime|_0)D$ is further small enough.
In this way we obtain a solution $w(t)$ to \eqref{IE} and, by Lemma \ref{lem-uni},
it is the only solution under the additional smallness \eqref{small2}
in the sense of the description of the latter part of Theorem \ref{main}.

Finally, the large time behavior \eqref{decay} and further continuity \eqref{sol-cl2} as well as \eqref{IC}
immediately follow from \eqref{3-est}, \eqref{Psi-conti2}--\eqref{decay-q} and \eqref{top-conti-2}--\eqref{top-decay}.
We have completed the proof Theorem \ref{main}.
\hfill
$\Box$

\medskip
\noindent
{\bf Acknowledgements.} The work of G.P.~Galdi is partially supported by National Science Foundation Grant DMS-2307811.
The work of T.~Hishida is partially supported by the Grant-in-aid for Scientific Research 22K03372 from JSPS. 

\medskip
\noindent
\centerline{\bf Declaration.}
\begin{itemize}
\item
The authors have no competing interests to declare that are relevant to the content of this
article.
\item Data sharing not applicable to this article as no datasets were generated or analyzed during
the current study.

\end{itemize}

\begin{flushleft}
Giovanni P. Galdi \\
Department of Mechanical Engineering and Materials Science \\
University of Pittsburgh \\
Pittsburgh PA15261 \\
USA \\
e-mail: galdi@pitt.edu \\
$\,$ \\
Toshiaki Hishida \\
Graduate School of Mathematics \\
Nagoya University \\
Nagoya 464-8602 \\
Japan \\
e-mail: hishida@math.nagoya-u.ac.jp
\end{flushleft}

\end{document}